\documentclass[11pt,a4paper]{article}
\usepackage{hyperref}
\usepackage[english]{babel}
\usepackage[utf8]{inputenc}
\usepackage[T1]{fontenc}
\usepackage{amsmath,amsthm,amssymb,amsfonts}
\numberwithin{equation}{section}
\usepackage{enumerate}
\usepackage{faktor}
\usepackage{dsfont}
\usepackage{scalerel,stackengine}
\usepackage{textcomp}
\usepackage{graphicx}
\usepackage{extarrows}
\usepackage{epigraph}
\usepackage{graphicx}
\usepackage{subcaption}
\usepackage{sidecap}
\usepackage{indentfirst}
\usepackage{tikz}
\usepackage{tikz-cd}
\usepackage{bm}
\usetikzlibrary{shapes.misc,arrows,matrix,patterns,decorations.markings,positioning}
\tikzset{cross/.style={cross out, draw=black, minimum size=2*(#1-\pgflinewidth), inner sep=0pt, outer sep=0pt},
cross/.default={4.5pt}}
\usepackage{pgfplots}
\usepackage{caption}
\usepackage{float}
\usepackage{color}
\usepackage{epstopdf}
\usepackage{epsfig}
\usepackage{stmaryrd}
\usepackage{color}
\usepackage{geometry}
\usepackage{multicol}
\usepackage{verbatim}
\usepackage{graphicx}
\usepackage{euscript}
\usepackage{amsmath}
\usepackage{amsfonts}
\usepackage{amssymb}
\usepackage{amsthm}
\usepackage{epsfig}
\usepackage{epstopdf}
\usepackage{url}
\usepackage{array}
\usepackage{amssymb}\usepackage{amscd}
\usepackage{epic}
\usepackage{tikz,underscore}
\usepackage{tikz-cd}
\usepackage{float}
\usepackage{tkz-euclide}

\numberwithin{equation}{section}

\DeclareMathOperator{\lk}{\ell k}

\renewcommand{\geq}{\geqslant}
\renewcommand{\leq}{\leqslant} 
\renewcommand{\epsilon}{\varepsilon}

\newcommand{\Z}{\mathbb{Z}}

\newcommand{\F}{\mathbb{F}}

\def\Z{\mathbb{Z}}

\def\L{\mathcal{L}}
\def\H{\mathbb{H}}
\def\F{\mathbb{F}}

\def\s{\mathfrak{s}}

\def\OS{Ozsv\'ath }

\DeclareFontFamily{U}{mathx}{\hyphenchar\font45}
\DeclareFontShape{U}{mathx}{m}{n}{
      <5> <6> <7> <8> <9> <10>
      <10.95> <12> <14.4> <17.28> <20.74> <24.88>
      mathx10
      }{}
\DeclareSymbolFont{mathx}{U}{mathx}{m}{n}
\DeclareFontSubstitution{U}{mathx}{m}{n}
\DeclareMathAccent{\widecheck}{0}{mathx}{"71}
\DeclareMathAccent{\wideparen}{0}{mathx}{"75}

\theoremstyle{plain}
\newtheorem{theorem}{Theorem}[section]
\newtheorem{lemma}[theorem]{Lemma}
\newtheorem{proposition}[theorem]{Proposition}

\newtheorem{corollary}[theorem]{Corollary}

\theoremstyle{definition}
\newtheorem{example}[theorem]{Example}
\newtheorem{remark}[theorem]{Remark}
\newtheorem{definition}[theorem]{Definition}

\pgfplotsset{compat=1.17}
\begin{document}
\title{Fibered and strongly quasi-positive $L$-space links}

\author{\scshape{Alberto Cavallo and Beibei Liu}\\ \\
 \footnotesize{Max Planck Institute for Mathematics,}\\
 \footnotesize{Bonn 53111, Germany}\\ \\ \small{cavallo@mpim-bonn.mpg.de}\\ \small{bbliumath@gmail.com}}
\date{}


\maketitle

\begin{abstract}
Every $L$-space knot is fibered and strongly quasi-positive, but this does not hold for $L$-space links. In this paper, we use the so called $H$-function, which is a concordance link invariant, to introduce a subfamily of fibered strongly quasi-positive $L$-space links. 
Furthermore, we present an infinite family of $L$-space links which are not quasi-positive.     
\end{abstract}

\section{Introduction}

It is well known that every $L$-space knot is fibered and strongly quasi-positive \cite{Ghiggini,Hedden,OS05}. Surprisingly, this does not hold for general $L$-space links. Not every $L$-space link is fibered, an example (\cite{LiuY}) is given by the 2-bridge links $b(10n+4,-5)$ with $2\leq n\leq6$, and this can be detected from the rank of $\widehat{HFL}$ on the top degree, see \cite{Ghiggini,Ni}. Not all fibered $L$-space links are strongly quasi-positive, or quasi-positive. For example, the Whitehead link is a fibered $L$-space link, but it is neither strongly quasi-positive nor quasi-positive. 

Indeed, $L$-space knots and $L$-space links are quite different in a lot of aspects. They are both defined in terms of large surgeries, but if there is one $L$-space surgery on this knot then the knot is an $L$-space knot, and all surgeries with surgery coefficient greater than that one are $L$-space surgeries. This also does not hold for $L$-space links (see \cite{LiuY}), and the $L$-space surgeries on $L$-space links are in general much more complicated than surgeries on knots, see \cite{GN2, Liu2, Sar}. 

The $H$-function $H_{\L}(\bm{s})$ is a link concordance invariant, defined on some $n$-dimensional lattice $\H(\L)$ where $n$ is the number of components of the link, and takes non-negative integer values, see \cite{GN}. Note that whether a link $\L$ is an $L$-space link does not depend on the orientation of $\L$. However, the $H$-function depends on the orientation. 

We say a link $\L$ is \emph{type (B)} if there is a lattice point $\bm{s}=(s_{1}, \cdots, s_{n})\in \H(\L)$ such that $H_{\L}(\bm{s}')=0$ if and  only if  $\bm{s}'\succeq \bm{s}$. Otherwise, we say the link $\L$ is \emph{type (A)}. If furthermore $\L$ is a type (B) $L$-space link then we call $\L$ \emph{special} when its $H$-function satisfies the following three equations:
 \begin{equation} 
  \label{equ1}
  s_i=g_3(L_i)+\dfrac{\lk(L_i,\L\setminus L_i)}{2} 
 \end{equation}
 for every $i=1,...,n$;
 \begin{equation}
 \label{equ2}
     H_{\L}(\bm{s}-\bm{1})=1
\end{equation}
        where $\bm{1}=\{1, \cdots, 1\};$
   
\begin{equation}
    \label{equ3}
     H_{\L}(x_1,\cdots,x_{n})=H_{\L}(\bar{x}_{1},\cdots,\bar{x}_{n})
 \end{equation}
 for every $\textbf x=(x_{1}, \cdots, x_{n})\in\H(\L)$  with  \[\sum_{i=1}^nx_i\geq\sum_{i=1}^ns_i-n\:,\] where $\bar{x}_{i}=\infty$ if $x_{i}\geq s_{i}$ and $\bar{x}_{i}=x_{i}$ otherwise. For example, the link $L7n1$ is a special $L$-space link, while the 2-component unlink is type (B) but not special, and their $H$-function is shown in Figure \ref{figure7}.
\begin{figure}[H]
\centering 
\begin{tikzpicture}[scale =.6]
\draw[help lines, color=black!30, dashed] (-4.9,-4.9) grid (4.9,4.9);
\draw[->, color=black, thick] (-5,0)--(5,0) node[right]{$s_{1}$};
\draw[->, color=black, thick] (0,-5)--(0,5) node[above]{$s_{2}$};
\node[color=black] at (0.2, 0) {1};
\node[color=black] at (0.2, 1) {1};
\node[color=black] at (0.2, 2) {1};
\node[color=black] at (0.1, 3) {1};
\node[color=black] at (0.1, 4) {1};
\node[color=black] at (0.1, 5) {$\vdots$};
\node[color=black] at (0.2, -1) {2};
\node[color=black] at (0.2, -2) {3};
\node[color=black] at (0.1, -3) {4};
\node[color=black] at (0.1, -4) {5};
\node[color=black] at (0.1, -5) {$\vdots$};

\node[color=black] at (1.1, 0) {1};
\node[color=black] at (1.2, 1) {1};
\node[color=black] at (1.2, 2) {0};
\node[color=black] at (1.1, 3) {0};
\node[color=black] at (1.1, 4) {0};
\node[color=black] at (1.1, 5) {$\vdots$};
\node[color=black] at (1.1, -1) {2};
\node[color=black] at (1.1, -2) {3};
\node[color=black] at (1.1, -3) {4};
\node[color=black] at (1.1, -4) {5};
\node[color=black] at (1.1, -5) {$\vdots$};

\node[color=black] at (2.1, 0) {1};
\node[color=black] at (2.1, 1) {1};
\node[color=black] at (2.1, 2) {0};
\node[color=black] at (2.1, 3) {0};
\node[color=black] at (2.1, 4) {0};
\node[color=black] at (2.1, 5) {$\vdots$};
\node[color=black] at (2.1, -1) {2};
\node[color=black] at (2.1, -2) {3};
\node[color=black] at (2.1, -3) {4};
\node[color=black] at (2.1, -4) {5};
\node[color=black] at (2.1, -5) {$\vdots$};

\node[color=black] at (3.1, 0) {1};
\node[color=black] at (3.1, 1) {1};
\node[color=black] at (3.1, 2) {0};
\node[color=black] at (3.1, 3) {0};
\node[color=black] at (3.1, 4) {0};
\node[color=black] at (3.1, 5) {$\vdots$};
\node[color=black] at (3.1, -1) {2};
\node[color=black] at (3.1, -2) {3};
\node[color=black] at (3.1, -3) {4};
\node[color=black] at (3.1, -4) {5};
\node[color=black] at (3.1, -5) {$\vdots$};

\node[color=black] at (4.1, 0) {1};
\node[color=black] at (4.1, 1) {1};
\node[color=black] at (4.1, 2) {0};
\node[color=black] at (4.1, 3) {0};
\node[color=black] at (4.1, 4) {0};
\node[color=black] at (4.1, 5) {$\vdots$};
\node[color=black] at (4.1, -1) {2};
\node[color=black] at (4.1, -2) {3};
\node[color=black] at (4.1, -3) {4};
\node[color=black] at (4.1, -4) {5};
\node[color=black] at (4.1, -5) {$\vdots$};

\node[color=black] at (5, 0) {$\cdots$};
\node[color=black] at (5, 1) {$\cdots$};
\node[color=black] at (5, 2) {$\cdots$};
\node[color=black] at (5, 3) {$\cdots$};
\node[color=black] at (5, 4) {$\cdots$};
\node[color=black] at (5, 5) {$\vdots$};
\node[color=black] at (5, -1) {$\cdots$};
\node[color=black] at (5, -2) {$\cdots$};
\node[color=black] at (5, -3) {$\cdots$};
\node[color=black] at (5, -4) {$\cdots$};
\node[color=black] at (5, -5) {$\vdots$};

\node[color=black] at (-0.9, 0) {2};
\node[color=black] at (-0.9, 1) {2};
\node[color=black] at (-0.9, 2) {2};
\node[color=black] at (-0.9, 3) {2};
\node[color=black] at (-0.9, 4) {2};
\node[color=black] at (-0.9, 5) {$\vdots$};
\node[color=black] at (-0.8, -1) {3};
\node[color=black] at (-0.8, -2) {3};
\node[color=black] at (-0.9, -3) {4};
\node[color=black] at (-0.9, -4) {5};
\node[color=black] at (-0.9, -5) {$\vdots$};

\node[color=black] at (-1.8, 0) {3};
\node[color=black] at (-1.9, 1) {3};
\node[color=black] at (-1.9, 2) {3};
\node[color=black] at (-1.9, 3) {3};
\node[color=black] at (-1.9, 4) {3};
\node[color=black] at (-1.9, 5) {$\vdots$};
\node[color=black] at (-1.8, -1) {4};
\node[color=black] at (-1.8, -2) {4};
\node[color=black] at (-1.9, -3) {5};
\node[color=black] at (-1.9, -4) {6};
\node[color=black] at (-1.9, -5) {$\vdots$};

\node[color=black] at (-2.9, 0) {4};
\node[color=black] at (-2.9, 1) {4};
\node[color=black] at (-2.9, 2) {4};
\node[color=black] at (-2.9, 3) {4};
\node[color=black] at (-2.9, 4) {4};
\node[color=black] at (-2.9, 5) {$\vdots$};
\node[color=black] at (-2.9, -1) {5};
\node[color=black] at (-2.9, -2) {5};
\node[color=black] at (-2.9, -3) {6};
\node[color=black] at (-2.9, -4) {7};
\node[color=black] at (-2.9, -5) {$\vdots$};

\node[color=black] at (-4, 0) {5};
\node[color=black] at (-4, 1) {5};
\node[color=black] at (-4, 2) {5};
\node[color=black] at (-4, 3) {5};
\node[color=black] at (-4, 4) {5};
\node[color=black] at (-4, 5) {$\vdots$};
\node[color=black] at (-4, -1) {6};
\node[color=black] at (-4, -2) {6};
\node[color=black] at (-4, -3) {7};
\node[color=black] at (-4, -4) {8};
\node[color=black] at (-4, -5) {$\vdots$};

\node[color=black] at (-5, 0) {$\cdots$};
\node[color=black] at (-5, 1) {$\cdots$};
\node[color=black] at (-5, 2) {$\cdots$};
\node[color=black] at (-5, 3) {$\cdots$};
\node[color=black] at (-5, 4) {$\cdots$};
\node[color=black] at (-5, 5) {$\vdots$};
\node[color=black] at (-5, -1) {$\cdots$};
\node[color=black] at (-5, -2) {$\cdots$};
\node[color=black] at (-5, -3) {$\cdots$};
\node[color=black] at (-5, -4) {$\cdots$};
\node[color=black] at (-5, -5) {$\vdots$};





\end{tikzpicture} 
\hspace{.5cm}
\begin{tikzpicture}[scale =.6]
\draw[help lines, color=black!30, dashed] (-4.9,-4.9) grid (4.9,4.9);
\draw[->, color=black, thick] (-5,0)--(5,0) node[right]{$s_{1}$};
\draw[->, color=black, thick] (0,-5)--(0,5) node[above]{$s_{2}$};
\node[color=black] at (0.2, 0) {0};
\node[color=black] at (0.2, 1) {0};
\node[color=black] at (0.2, 2) {0};
\node[color=black] at (0.1, 3) {0};
\node[color=black] at (0.1, 4) {0};
\node[color=black] at (0.1, 5) {$\vdots$};
\node[color=black] at (0.2, -1) {1};
\node[color=black] at (0.2, -2) {2};
\node[color=black] at (0.1, -3) {3};
\node[color=black] at (0.1, -4) {4};
\node[color=black] at (0.1, -5) {$\vdots$};

\node[color=black] at (1.1, 0) {0};
\node[color=black] at (1.2, 1) {0};
\node[color=black] at (1.2, 2) {0};
\node[color=black] at (1.1, 3) {0};
\node[color=black] at (1.1, 4) {0};
\node[color=black] at (1.1, 5) {$\vdots$};
\node[color=black] at (1.1, -1) {1};
\node[color=black] at (1.1, -2) {2};
\node[color=black] at (1.1, -3) {3};
\node[color=black] at (1.1, -4) {4};
\node[color=black] at (1.1, -5) {$\vdots$};

\node[color=black] at (2.1, 0) {0};
\node[color=black] at (2.1, 1) {0};
\node[color=black] at (2.1, 2) {0};
\node[color=black] at (2.1, 3) {0};
\node[color=black] at (2.1, 4) {0};
\node[color=black] at (2.1, 5) {$\vdots$};
\node[color=black] at (2.1, -1) {1};
\node[color=black] at (2.1, -2) {2};
\node[color=black] at (2.1, -3) {3};
\node[color=black] at (2.1, -4) {4};
\node[color=black] at (2.1, -5) {$\vdots$};

\node[color=black] at (3.1, 0) {0};
\node[color=black] at (3.1, 1) {0};
\node[color=black] at (3.1, 2) {0};
\node[color=black] at (3.1, 3) {0};
\node[color=black] at (3.1, 4) {0};
\node[color=black] at (3.1, 5) {$\vdots$};
\node[color=black] at (3.1, -1) {1};
\node[color=black] at (3.1, -2) {2};
\node[color=black] at (3.1, -3) {3};
\node[color=black] at (3.1, -4) {4};
\node[color=black] at (3.1, -5) {$\vdots$};

\node[color=black] at (4.1, 0) {0};
\node[color=black] at (4.1, 1) {0};
\node[color=black] at (4.1, 2) {0};
\node[color=black] at (4.1, 3) {0};
\node[color=black] at (4.1, 4) {0};
\node[color=black] at (4.1, 5) {$\vdots$};
\node[color=black] at (4.1, -1) {1};
\node[color=black] at (4.1, -2) {2};
\node[color=black] at (4.1, -3) {3};
\node[color=black] at (4.1, -4) {4};
\node[color=black] at (4.1, -5) {$\vdots$};

\node[color=black] at (5, 0) {$\cdots$};
\node[color=black] at (5, 1) {$\cdots$};
\node[color=black] at (5, 2) {$\cdots$};
\node[color=black] at (5, 3) {$\cdots$};
\node[color=black] at (5, 4) {$\cdots$};
\node[color=black] at (5, 5) {$\vdots$};
\node[color=black] at (5, -1) {$\cdots$};
\node[color=black] at (5, -2) {$\cdots$};
\node[color=black] at (5, -3) {$\cdots$};
\node[color=black] at (5, -4) {$\cdots$};
\node[color=black] at (5, -5) {$\vdots$};

\node[color=black] at (-0.9, 0) {1};
\node[color=black] at (-0.9, 1) {1};
\node[color=black] at (-0.9, 2) {1};
\node[color=black] at (-0.9, 3) {1};
\node[color=black] at (-0.9, 4) {1};
\node[color=black] at (-0.9, 5) {$\vdots$};
\node[color=black] at (-0.8, -1) {2};
\node[color=black] at (-0.8, -2) {3};
\node[color=black] at (-0.9, -3) {4};
\node[color=black] at (-0.9, -4) {5};
\node[color=black] at (-0.9, -5) {$\vdots$};

\node[color=black] at (-1.8, 0) {2};
\node[color=black] at (-1.9, 1) {2};
\node[color=black] at (-1.9, 2) {2};
\node[color=black] at (-1.9, 3) {2};
\node[color=black] at (-1.9, 4) {2};
\node[color=black] at (-1.9, 5) {$\vdots$};
\node[color=black] at (-1.8, -1) {3};
\node[color=black] at (-1.8, -2) {4};
\node[color=black] at (-1.9, -3) {5};
\node[color=black] at (-1.9, -4) {6};
\node[color=black] at (-1.9, -5) {$\vdots$};

\node[color=black] at (-2.9, 0) {3};
\node[color=black] at (-2.9, 1) {3};
\node[color=black] at (-2.9, 2) {3};
\node[color=black] at (-2.9, 3) {3};
\node[color=black] at (-2.9, 4) {3};
\node[color=black] at (-2.9, 5) {$\vdots$};
\node[color=black] at (-2.9, -1) {4};
\node[color=black] at (-2.9, -2) {5};
\node[color=black] at (-2.9, -3) {6};
\node[color=black] at (-2.9, -4) {7};
\node[color=black] at (-2.9, -5) {$\vdots$};

\node[color=black] at (-4, 0) {4};
\node[color=black] at (-4, 1) {4};
\node[color=black] at (-4, 2) {4};
\node[color=black] at (-4, 3) {4};
\node[color=black] at (-4, 4) {4};
\node[color=black] at (-4, 5) {$\vdots$};
\node[color=black] at (-4, -1) {5};
\node[color=black] at (-4, -2) {6};
\node[color=black] at (-4, -3) {7};
\node[color=black] at (-4, -4) {8};
\node[color=black] at (-4, -5) {$\vdots$};

\node[color=black] at (-5, 0) {$\cdots$};
\node[color=black] at (-5, 1) {$\cdots$};
\node[color=black] at (-5, 2) {$\cdots$};
\node[color=black] at (-5, 3) {$\cdots$};
\node[color=black] at (-5, 4) {$\cdots$};
\node[color=black] at (-5, 5) {$\vdots$};
\node[color=black] at (-5, -1) {$\cdots$};
\node[color=black] at (-5, -2) {$\cdots$};
\node[color=black] at (-5, -3) {$\cdots$};
\node[color=black] at (-5, -4) {$\cdots$};
\node[color=black] at (-5, -5) {$\vdots$};





\end{tikzpicture}
\caption{The $H$-functions of $L7n1$ (left) and $\bigcirc_2$ (right).} \label{figure7}
\end{figure}
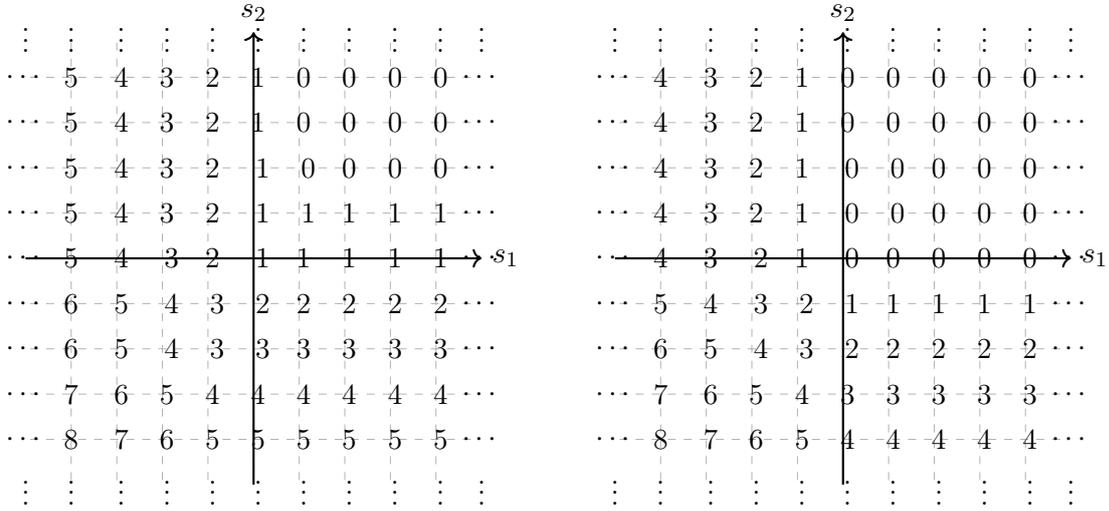

The latter conditions are actually not very restrictive; for example, all $L$-space knots and all algebraic links are special $L$-space links.
\begin{proposition}
\label{algebralinks}
Every algebraic link is a special $L$-space link. In particular, this holds for positive, coherently oriented torus links $T_{p,q}$ for every $1\leq p\leq q$.
\end{proposition}

Furthermore, suppose $\L=L_1\cup \cdots \cup L_{n}$ is a special $L$-space link; let $p_i, q_i$ be coprime positive integers where $1\leq i \leq n$, and let $L_{p_{i}, q_{i}}$ denote the $(p_{i}, q_{i})$-cable of $L_{i}$. Then the link $\L_{cab}=L_{p_{1}, q_{1}}\cup \cdots \cup L_{p_{n}, q_{n}}$ is also an $L$-space link if $q_{i}/p_{i}$ is sufficiently large for each $1\leq i \leq n$, \cite[Proposition 2.8]{BG}, and it is special \cite[Theorem 4.6]{Liu}.

It is in general very hard to study the type (A) links, but we will concentrate on 2-component type (B) $L$-space links in this paper. Note that the definitions of type (A) and (B) for $L$-space links are slightly different from the ones in \cite{Liu}. This is simply for different purposes.

\begin{theorem}
\label{stypeb}
Suppose that $\L$ is a special $L$-space  link. Then $\L$ is fibered and strongly quasi-positive.
\end{theorem}

We do not know whether in general the converse holds or not, but for 2-component $L$-space links, we have the following theorem.

\begin{theorem}
\label{2com}
A 2-component $L$-space link is fibered and strongly quasi-positive if and only if it is special. 
\end{theorem}
An application of Theorem \ref{stypeb} is a proof of the following corollary, see \cite{Cavallo18,Cavallo19} for the definitions of $\tau$ and $\nu^+$ for links.
\begin{corollary}
 \label{cor:genus}
 Suppose that $\L$ is a special $L$-space link. Then we have \[g_3(\L)+n-1=\sum_{i=1}^ng_3(L_i)+\sum_{i<j}\lk(L_i,L_j)=\tau(\L)=\nu^+(\L)\:.\]
 In particular, when $\L$ has two unknotted components this equation implies $\lk(L_1,L_2)>0$. 
\end{corollary}

It is worth to mention that in general, it is not easy to determine whether a link is an $L$-space link. One can apply Theorem \ref{stypeb} to prove a link is not an $L$-space link. For example, Boileau, Boyer and Gordon studied strongly quasi-positive pretzel links  in \cite{BBG}. In order to check whether a pretzel link is an $L$-space link, one can compute its $H$-function, see \cite{Liu2}, to see whether it is special. For example, the pretzel link $b(-2, 3, 8)$   is not strongly quasi-positive \cite[Proposition 8.1]{BBG}; if we assume that it is an $L$-space link then it is special, which contradicts Theorem \ref{stypeb}. For the computation of its $H$-function, see \cite{Liu3}. 

Consider the family of 2-bridge links $b(4k^{2}+4k, -2k-1)$ where $k$ is a positive integer as in Figure \ref{two bridge}; they are type (A) 2-component $L$-space links. 
We are able to prove the following result.
\begin{proposition}
\label{prop:quasi}
 The family of 2-bridge $L$-space links $b(4k^{2}+4k, -2k-1)$ with $k\geq1$ contains no quasi-positive link. In particular, the Whitehead links are not quasi-positive. 
\end{proposition}
\begin{remark}
Similarly, one can prove that the Borromean rings are also not quasi-positive. 
\end{remark}
\begin{figure}[H]
\centering
\includegraphics[width=4in]{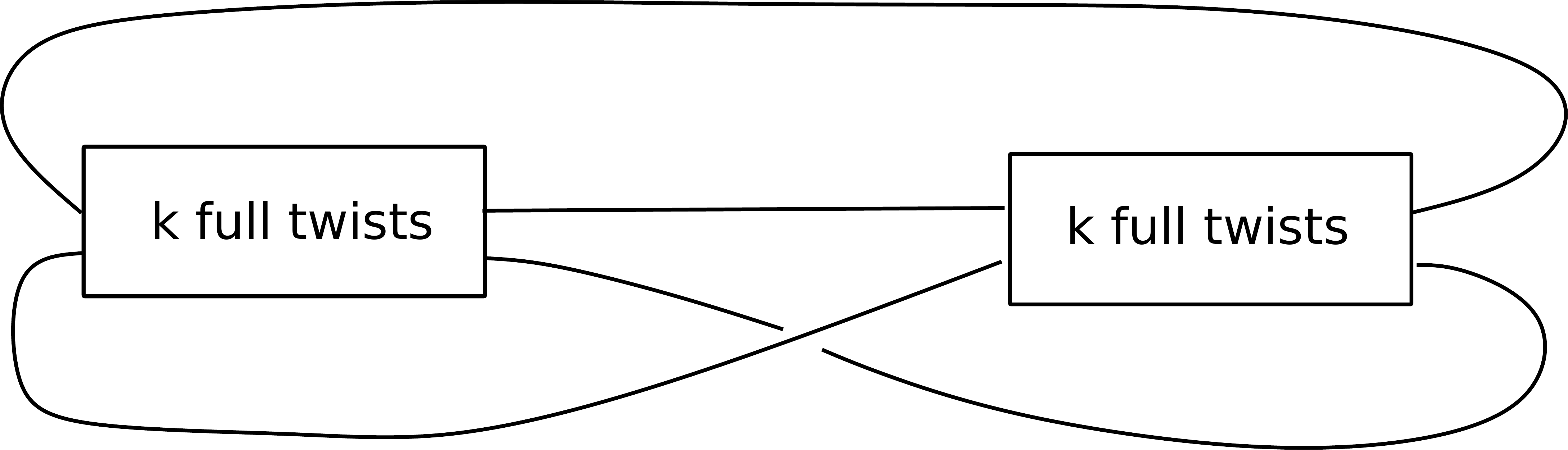}
\caption{ Two-bridge link $b(4k^2+4k, -2k-1).$   \label{two bridge}}
\end{figure}
The key ingredient of the proofs of Theorem \ref{stypeb} and Theorem \ref{2com} is the characterization of strongly quasi-positive fibered links via the $\tau$-invariant proved by the first author in \cite{Cavallo20}:
a fibered link $\L$ in $S^{3}$ is strongly quasi-positive if and only if $\tau(\L)=g_{3}(\L)+n-1$. 

The paper is organized as follows: in Subsection \ref{link floer} and Section \ref{sec:the H-function}, we review link Floer homology, the definitions of $L$-space links and the $H$-function. In Subsection \ref{sec:colla}, we review  the collapsed link Floer complex and the $\tau$ invariant for links. In Section \ref{sec:proof}, we give the proofs of our main results: Theorem \ref{stypeb} and Theorem \ref{2com}. \\

{\bf Acknowledgement.} We are grateful to the Max Planck Institute for Mathematics in Bonn for its hospitality and financial supports. We also want to thank the referee for its suggestions.

\section{Basics of link Floer homology}
\label{back}
\subsection{The link Floer homology groups \texorpdfstring{$HFL^{-}(\L)$}{HFL-(L)} and \texorpdfstring{$\widehat{HFL}(\L)$}{HFLhat(L)}}
\label{link floer}
\OS  and Szab\'o proved that there are chain complexes $CF^{-}(M), \widehat{CF}(M)$ associated to an admissible multi-pointed Heegaard diagram for a closed oriented connected 3-manifold $M$ \cite{OS1}, and these give 3-manifold invariants $HF^{-}(M)$ and $\widehat{HF}(M)$. A null-homologous link $\L=L_{1}\cup \cdots \cup L_{n}$ in $M$ defines a filtration on the chain complex $CF^{-}(M)$. For links in $S^{3}$, this filtration is indexed by an $n$-dimensional lattice $\H(\L)$ which is defined as follows.

\begin{definition}
For an oriented link $\L=L_{1}\cup \cdots \cup L_{n}\subset S^{3}$, define $\H(\L)$ to be the affine lattice over $\Z^{n}$:
\[H(\L)=\bigoplus_{i=1}^{n} \H_{i}(\L), \quad \H_{i}(\L)=\Z+\dfrac{\lk(L_{i}, \L\setminus L_{i})}{2}\]
where $\lk(L_{i}, \L\setminus L_{i})$ denotes the linking number of $L_{i}$ and $\L\setminus L_{i}$. 
\end{definition}

Given $\bm{s}=(s_{1}, \cdots, s_{n})\in \H(\L)$, the \emph{generalized Heegaard Floer complex} $A_*^{-}(\L, \bm{s})\subset CFL_*^{-}(\mathcal L)=CF^{-}(S^{3})$ is the $\F[U_1,...,U_n]$-module defined as the subcomplex of $CF^{-}(S^{3})$ corresponding to the Alexander filtration indexed by $\bm{s}\in\H(\L)$ \cite{MO}. 

We recall that the filtration is increasing in the following sense. Let us consider $\textbf v,\textbf w\in\H(\mathcal L)$; we write $\textbf v\preceq\textbf  w$ if $v_i\leq w_i$ for every $i=1,...,n$. In the same way, we say that $\textbf v\prec\textbf w$ when $\textbf v\preceq \textbf w$ and $\textbf v\neq\textbf w$.
Hence, we have \[A_*^-(\mathcal L,\textbf v)\subset A_*^-(\mathcal L,\textbf w)\] whenever $\textbf v\preceq \textbf w\in\H(\mathcal L)$.

It is known \cite{MO} that the actions of the $U_i$'s on $CFL_*^{-}(\mathcal L)$ are all homotopic; hence, they coincide in the homology group $HF^-(S^3)$, inducing a natural structure of $\F[U]$-module.
Since $A^-(\mathcal L,\textbf s)=CFL^-(\mathcal L)$ when $\textbf s$ is big enough, the total homology of $CFL_*^-(\mathcal L)$ is always isomorphic to $HF^-_*(S^3)\cong\F[U]_0$ for every link. On the other hand, the link Floer homology group $HFL_*^{-}(\L, \bm{s})$ is defined as the homology of the bigraded complex associated to the filtration and it depends on the isotopy type of $\L$; more specifically, we have \[HFL_*^-(\L,\textbf s)=H_*\left(\faktor{A^-(\L,\textbf s)}{\sum_{\textbf u\prec \textbf s}A^-(\L,\textbf u)}\right)\:,\] which is also an $\F[U_1,...,U_n]$-module. For more details, see \cite{MO, SZ3}. Besides the link invariant  $HFL^{-}(\L, \bm{s})$,  Ozsv\'ath and Szab\'o also associated the multi-graded link invariant $\widehat{HFL}(\L, \bm{s})$ to links $\L\subset S^{3}$ which  is defined as follows \cite{ Hom, SZ3}:
\[\widehat{HFL}(\L, \bm{s})=H_{\ast}\left(A^{-}(\L, \bm{s})/\left[\sum_{i=1}^{n}A^{-}(\bm{s}-\bm{e}_{i})\oplus \sum_{i=1}^{n}U_{i}A^{-}(\bm{s}+\bm{e}_{i})\right]\right)\:.\] 
The latter is a finite dimensional $\F$-vector space.

\subsection{The collapsed Alexander filtration}
\label{sec:colla}
Following the notation in \cite{Cavallo18}, we introduce the complex $cCFL_*^-(\L)$ by collapsing the variables $U_1,...,U_n$ in $CFL^-_*(\L)$ to $U$. The Alexander filtration is also collapsed accordingly and we then obtain the subcomplexes $cA_*^-(\L,s)$, where $s=s_1+...+s_n\in\Z$. Obviously, the filtration is still increasing.  

The complex $cCFL^-(\L)$ can be identified with $CF^-(S^3,n)$ and its differential is also gotten from $CFL^-(\L)$ by collapsing the $U_i$'s to $U$. For this reason we have that its total homology is now isomorphic to $\F[U]^{2^{n-1}}$ for every link with $n$-components; more specifically, one has \[H_*(CF^-(S^3,n))\cong HF_{*-\frac{n-1}{2}}^-((S^1\times S^2)^{\#(n-1)},\mathfrak t_0)\cong HF_0^-(S^3)\otimes\left(\F[U]_{-1}\oplus\F[U]_0\right)^{\otimes\:n-1}\] with $\mathfrak t_0$ being the unique torsion Spin$^c$ structure on $(S^1\times S^2)^{\#(n-1)}$, see \cite{Cavallo18,Cavallo19,SZ3}. Using this complex the first author in \cite{Cavallo18} defined the invariant $\nu^+(\L)$, a link version of the invariant from \cite{HW}, as the minimal integer $s$ such that the inclusion \[H_0\left(cA^-(\L,s)\right)\lhook\joinrel\longrightarrow HF_0^-(S^3)\otimes\left(\F[U]_{-1}\oplus\F[U]_0\right)^{\otimes\:n-1}\] is non-trivial, meaning its image is not $\{0\}$. We recall that \[\dim_{\F[U]}HF^-(S^3)=\dim_{\F[U]}HF^-_0(S^3)=\dim_{\F}\widehat{HF}(S^3)=\dim_{\F}\widehat{HF}_0(S^3)=1\:;\] and $\dim_{\F[U]}$ denotes the rank of the free part of an $\F[U]$-module.
Later, in \cite{Cavallo19} it was proved that $\nu^+(\L)$ is a concordance link invariant.

The last homology group that we define in this paper is the hat version of of $cHFL^-(\L)$. We say that $\widehat{CFL}_*(\L)$ is the complex obtained from $CFL^-_*(\L)$ by setting $U_1=...=U_n=0$, which means it is a finite dimensional $\F$-vector space, together with the (collapsed) Alexander filtration on it, denoted by $\widehat A_*(\L,s)$. In other words, one has \[\widehat{CFL}(\L)=\widehat{CF}(S^3,n)=\widehat{CF}((S^1\times S^2)^{\#(n-1)})\:.\]
As in \cite{Cavallo18} we recall that the invariant $\tau(\L)$ is the minimal $s\in\Z$ such that the inclusion \[H_0\left(\widehat A(\L,s)\right)\lhook\joinrel\longrightarrow H_0\left(\widehat{CFL}(\L)\right)=\widehat{HF}_0(S^3)\otimes\left(\F_{-1}\oplus\F_0\right)^{\otimes\:n-1}\] is non-trivial.
The integer $\tau(\L)$ is also a concordance invariant; moreover, we have the following lower bound for the slice genus: \[\tau(\L)\leq\nu^+(\L)\leq g_4(\L)+n-1\leq g_3(\L)+n-1\:.\] The following result was proved in \cite{Cavallo20}. We recall that a link is said strongly quasi-positive if it can be written as closure of the composition of $d$-braids of the form 
\[(\sigma_i\cdot\cdot\cdot\sigma_{j-2})\sigma_{j-1}(\sigma_i\cdot\cdot\cdot\sigma_{j-2})^{-1}\quad\text{ for some }\quad d\geq j\geq i+2\geq3\] or \[\sigma_i\quad\text{ for }\quad i=1,...,d-1\:,\]
where $\sigma_1,...,\sigma_{d-1}$ are the Artin generators of the $d$-braids group.
\begin{theorem}\emph{\cite[Theorem 1.1]{Cavallo20}}
 \label{thm:fibered_SQP}
 A fibered link $\L$ in $S^3$ is strongly quasi-positive if and only if $\tau(\L)=\nu^+(\L)=g_3(\L)+n-1$, where $g_3(\L)$ is the Seifert genus of $\L$.
\end{theorem}
Finally, we call $\widehat{HFL}(\L)$ the bigraded object associated to $\widehat{CFL}(\L)$. For any $s\in \mathbb{Z}$, one defines: 
 \[\widehat{HFL}_*(\L,s)=H_*\left(\faktor{\widehat A(\L,s)}{\widehat A(\L,s-1)}\right)\] and this is an $\F$-vector space, which also depends on $\L$. Note that 
 \[\widehat{HFL}(\L, s)=\bigoplus_{\{\bm{s}\in \H(\L)\mid s_{1}+\cdots+s_{n}=s\}} \widehat{HFL}(\L, \bm{s})\:.\]
 
 We call $s_{\text{top}}$ the maximal Alexander grading $s$ such that $\widehat{HFL}_*(\L,s)\neq\{0\}$.
\begin{theorem}\emph{\cite{Ghiggini,Ni}}
 \label{thm:fibered}
 A link $\L$ is fibered if and only if $\dim_{\F}\widehat{HFL}_*(\L,s_{\text top})=1$.
\end{theorem}
We also recall that for an $n$-component fibered link $\L$ one has $s_{\text{top}}=g_3(\L)+n-1$. This is a classical result which follows from the fact that every Seifert surface for a fibered link is connected and if it realizes the minimal Seifert genus then it is isotopic to a fiber, see \cite{Kawauchi,Ni09}. A more detailed proof is in \cite{Cavallo20}.

\section{The \texorpdfstring{$H$}{H}-function and \texorpdfstring{$L$}{L}-space links}
\label{sec:the H-function}
\subsection{Properties of the \texorpdfstring{$H$}{H}-function}
By the large surgery theorem \cite[Theorem 12.1]{MO}, the homology of $A^{-}(\L, \bm{s})$ is isomorphic to the Heegaard Floer homology of a large surgery on the link $\L$, equipped with some Spin$^{c}$ structure which depends on $\bm s$, as an $\F[U]$-module. Thus the homology of $A^{-}(\L, \bm{s})$ is a direct sum  of one copy of $\F[U]$ and some $U$-torsion submodule. 

\begin{definition}\cite[Definition 3.9]{BG}
For an oriented link $\L\subset S^{3}$, we define the $H$-function $H_{\L}(\bm{s})$ by saying that  $-2H_{\L}(\bm{s})$ is the maximal homological degree of the free part of $H_{\ast}(A^{-}(\L, \bm{s}))$ where $\bm{s}\in \H(\L)$. 
\end{definition}

We list several properties of the $H$-function as follows.

\begin{lemma}\emph{\cite[Proposition 3.10]{BG}}
\label{growth control}
For an oriented link $\L\subset S^{3}$, the $H$-function $H_{\L}(\bm{s})$ takes non-negative values, and $H_{\L}(\bm{s}-\bm{e}_{i})=H_{\L}(\bm{s})$ or $H_{\L}(\bm{s}-\bm{e}_{i})=H_{\L}(\bm{s})+1$ where $\bm{s}\in \H(\L)$. 

\end{lemma}

\begin{lemma}\emph{\cite[Proposition 3.12]{BG}}
\label{infinity}
For an oriented link $\L=L_{1}\cup \cdots \cup L_{n}\subset S^{3}$ and $\bm{s}=(s_{1}, \cdots, s_{n})\in \H(\L)$, one has
\[H_{\L}(s_{1},\cdots, s_{n-1}, \infty)=H_{\L\setminus L_{n}}(s_{1}-\lk(L_{1}, L_{n})/2, \cdots, s_{n-1}-\lk(L_{n-1}, L_{n})/2 )\:,\]
where $\lk(L_{i}, L_{n})$ denotes the linking number of $L_{i}$ and $L_{n}$ for $i=1, 2, \cdots, n-1$.
\end{lemma}

\begin{remark}
We use the convention that $H_{\L}(\infty, \cdots, \infty)=0$.
\end{remark}

In general, it is very hard to compute the $H$-function for links in $S^{3}$. However, for the family of $L$-space links, its $H$-function can be computed from its Alexander polynomials \cite{BG}. We first review the definition of $L$-space links, introduced by \OS and Szab\'o in \cite{OS05}.
\begin{definition}
A 3-manifold $Y$ is an $L$-space if it is a rational homology sphere and its Heegaard Floer homology has minimal possible rank: for any Spin$^{c}$-structure $\s$, $\widehat{HF}(Y, \s)\cong\F$ and $HF^{-}(Y, \s)$ is a free $\F[U]$-module of rank 1. 
\end{definition}

\begin{definition}\cite{GN, Liu}
An  $n$-component link $\L\subset S^{3}$ is an $L$-space link if there exists $\bm{p}\in \Z^{n}$ with $0<p_i$ for every $i=1,...,n$ such that the surgery manifold $S^{3}_{\bm{q}}(\L)$ is an $L$-space for any $\bm{q}\succeq \bm{p}$.
\end{definition}

$L$-space links have the following properties.

\begin{theorem}\emph{\cite{Liu}}
\label{L-space link pro}
We have that
\begin{enumerate}[a)]
    \item every sublink of an $L$-space link is an $L$-space link; 
    \item a link is an $L$-space link if and only if for all $\bm{s}$ one has $H_{\ast}(A^{-}(\L, \bm{s}))\cong\F[U]$.
\end{enumerate}
\end{theorem}

For $L$-space links, the $H$-function can be computed from the multi-variable Alexander polynomial.
Indeed, by (b) and the inclusion-exclusion formula, one can write
\begin{equation}
\label{computation of h-function 1}
\chi(HFL^{-}(\L, \bm{s}))=\sum_{B\subset \lbrace 1, \cdots, n \rbrace}(-1)^{|B|-1}H_{\L}(\bm{s}-\bm{e}_{B}),
\end{equation}
as in \cite[Equation (3.14)]{BG}. The Euler characteristic $\chi(HFL^{-}(\L, \bm{s}))$ was computed in \cite{OS:linkpoly},
\begin{equation}
\label{computation 3}
\tilde{\Delta}(t_{1}, \cdots, t_{n})=\sum_{\bm{s}\in \H(\L)}\chi(HFL^{-}(\L, \bm{s}))t_{1}^{s_{1}}\cdots t_{n}^{s_{n}}
\end{equation}
where $\bm{s}=(s_{1}, \cdots, s_{n})$, and
\begin{equation}
\label{mva}
\widetilde{\Delta}_{\L}(t_{1}, \cdots, t_{n}): = \left\{
        \begin{array}{ll}
           (t_{1}\cdots t_{n})^{1/2} \Delta_{\L}(t_{1}, \cdots, t_{n}) & \quad \textup{if } n >1, \\
            \Delta_{\L}(t)/(1-t^{-1}) & \quad  \textup{if } n=1. 
        \end{array}
    \right. 
\end{equation}

\begin{remark}
Here we expand the rational function as power series in $t^{-1}$, assuming that the exponents are bounded in positive direction. The Alexander polynomials are normalized so that they are symmetric about the origin. This still leaves out the sign ambiguity which can be resolved for $L$-space links by requiring that $H(s)\ge 0$ for all $s$.
\end{remark}

One can regard Equation \eqref{computation of h-function 1} as a system of linear equations for $H(s)$ and solve it explicitly 
using the values of the $H$-function for sublinks as the boundary conditions. We refer to \cite{BG,GN} for general formulas. The explicit formula for links with one and two components can be found in \cite{GLM}. 

\begin{example}\cite[Lemma 2.11]{Liu2}
\label{knotgenus}
For an $L$-space knot $K$ one has $H_{K}(s)=0$ if and only if $s\geq g_{3}(K)$. 
\end{example}

\begin{example}
\label{wh H}
The (symmetric) Alexander polynomial of the Whitehead link equals 
\[
\Delta(t_1,t_2)=-(t_1^{1/2}-t_1^{-1/2})(t_2^{1/2}-t_2^{-1/2}),
\]
so
\[
\widetilde{\Delta}(t_1,t_2)=(t_1t_2)^{1/2}\Delta(t_1,t_2)=-(t_1-1)(t_2-1).
\]
The values of the $H$-function are in Figure \ref{Whitehead}: note that the link is type (A).
\begin{figure}
\begin{center}
\begin{tikzpicture}
\draw (1,0)--(1,5);
\draw (2,0)--(2,5);
\draw (3,0)--(3,5);
\draw (4,0)--(4,5);
\draw (0,1)--(5,1);
\draw (0,2)--(5,2);
\draw (0,3)--(5,3);
\draw (0,4)--(5,4);
\draw (0.5,4.5) node {2};
\draw (1.5,4.5) node {1};
\draw (2.5,4.5) node {0};
\draw (3.5,4.5) node {0};
\draw (4.5,4.5) node {0};
\draw (0.5,3.5) node {2};
\draw (1.5,3.5) node {1};
\draw (2.5,3.5) node {0};
\draw (3.5,3.5) node {0};
\draw (4.5,3.5) node {0};
\draw (0.5,2.5) node {2};
\draw (1.5,2.5) node {1};
\draw (2.5,2.5) node {1};
\draw (3.5,2.5) node {0};
\draw (4.5,2.5) node {0};
\draw (0.5,1.5) node {3};
\draw (1.5,1.5) node {2};
\draw (2.5,1.5) node {1};
\draw (3.5,1.5) node {1};
\draw (4.5,1.5) node {1};
\draw (0.5,0.5) node {4};
\draw (1.5,0.5) node {3};
\draw (2.5,0.5) node {2};
\draw (3.5,0.5) node {2};
\draw (4.5,0.5) node {2};
\draw [->,dotted] (0,2.5)--(5,2.5);
\draw [->,dotted] (2.5,0)--(2.5,5);
\draw (5,2.7) node {$s_1$};
\draw (2.3,5) node {$s_2$};
\end{tikzpicture}
\end{center}
\caption{The $H$-function of the Whitehead link.}
\label{Whitehead}
\end{figure}
\end{example}

\subsection{Spectral sequences}
\label{sec:computation}
For an  $L$-space link $\L$, there exist spectral sequences converging to $HFL^{-}(\L)$ and $\widehat{HFL}(\L)$ respectively \cite{Hom, GN}. 

\begin{proposition}\emph{\cite[Theorem 1.4]{GN}}
\label{spectral sequence 1}
For an oriented $L$-space link $\L\subset  S^{3}$ with n components and $\bm{s}\in \H(\L)$, there exists a spectral sequence with $E_{\infty}=HFL^{-}(\L, \bm{s})$ and 
$$E_{1}=\bigoplus_{B\subset \lbrace 1, \cdots, n \rbrace} H_{\ast}(A^{-}(\L, \bm{s}-\bm{e}_{B})), $$
where the differential in $E_{1}$ is induced by inclusions.

\end{proposition}
\begin{remark}
\label{differential}
Precisely, the differential $\partial_{1}$ in the $E_{1}$-page is 
$$\partial_{1}(z(\bm{s}-\bm{e}_{B}))=\sum_{i\in B}U^{H_{\L}(\bm{s}-\bm{e}_{B})-H_{\L}(\bm{s}-\bm{e}_{B}+\bm{e}_{i})}z(\bm{s}-\bm{e}_{B}+\bm{e}_{i}),$$
where $z(\bm{s}-\bm{e}_{B})$ denotes the unique generator in $H_{\ast}(A^{-}(\L, \bm{s}-\bm{e}_{B}))$ with the homological grading $-2H_{\L}(\bm{s}-\bm{e}_{B})$. 

\end{remark}

\begin{proposition}\emph{\cite[Proposition 3.8]{Hom}}
\label{spectral sequence 2}
For an $L$-space link $\L\subset S^{3}$ with n components and $\bm{s}\in \H(\L)$, there exists a spectral sequence whose $E_{\infty}$ page is $\hat{E}_{\infty}=\widehat{HFL}(\L, \bm{s})$ and the  $E_{1}$ page is
$$\hat{E}_{1}=\bigoplus_{B\subset \lbrace 1, \cdots, n \rbrace}  HFL^{-}(\L, \bm{s}+\bm{e}_{B}).$$

\end{proposition}

In the rest of this section, we review the explicit computation of $HFL^{-}(\L, \bm{s})$ for 2-component $L$-space links $\L$. Indeed, in this case the spectral sequence in Proposition \ref{spectral sequence 1} comes from the following \emph{iterated cone} complex.
\begin{lemma}\emph{\cite[Lemma 2.8]{GN}}
For any  $(s_{1}, s_{2})\in \H(\L)$, the chain complex $CFL^{-}(s_{1}, s_{2})$ of the  $L$-space link $L=L_{1}\cup L_{2}$ is quasi-isomorphic to the iterated cone complex:
\[ \left[  
\begin{tikzcd}
A^{-}(s_{1}-1, s_{2}) \arrow{r}{i_{1}}  & A^{-}(s_{1}, s_{2})  \\%
A^{-}(s_{1}-1, s_{2}-1) \arrow{r}{i_{1}} \arrow[swap]{u}{i_{2}} & A^{-}(s_{1}, s_{2}-1) \arrow[swap]{u}{i_{2}}
 \end{tikzcd}
\right]   \] 
where $i_{1}$ and $i_{2}$ are  inclusion maps in \emph{\cite[Lemma 2.4]{GN}}. 
\end{lemma}


The spectral sequence collapses at its $E_{2}$-page \cite[Theorem 2.9]{GN}. Its $E_1$ page is in Figure \ref{mappingcone}, with the differential $d_1$ induced from inclusion $i_1, i_2$. Note that $d_{2}$ changes the homological grading by an odd integer, but $\F[U]$ has even grading, which implies that $d_2=0$. Hence, to compute $HFL^{-}(\L, \bm{s})$, we just need to consider $d_1$. 
\begin{remark}
\label{remark:total}
The Maslov grading is the sum of the homological grading and  the cube grading.
\end{remark}


 
 \begin{figure}[H]
 \centering
\[ \begin{tikzcd}
\mathbb{F}[U][-2H_{\L}(s_{1}-1, s_{2})][b] \arrow{r}{i_{1}}  & \mathbb{F}[U][-2H_{\L}(s_{1}, s_{2})][a] \\%
\mathbb{F}[U][-2H_{\L}(s_{1}-1, s_{2}-1)][c] \arrow{r}{i_{1}} \arrow[swap]{u}{i_{2}} & \mathbb{F}[U][-2h_{\L}(s_{1}, s_{2}-1)][d] \arrow[swap]{u}{i_{2}}
\end{tikzcd}
\]
\caption{$E_1$-page.}
\label{mappingcone}
\end{figure}
Let  $a, b, c, d$ denote the generators in $\mathbb{F}[U][-2H_{\L}(s_{1}, s_{2})],\mathbb{F}[U][-2H_{\L}(s_{1}-1, s_{2})],  \mathbb{F}[U][-2H_{\L}$ $(s_{1}-1, s_{2}-1)]$, and $\mathbb{F}[U][-2H_{\L}(s_{1}, s_{2}-1)]$, respectively. Let $h=H(s_{1}, s_{2})$. By Lemma \ref{growth control}, there are 6 cases for the $H$-function corresponding to the mapping cone. 

Based on the $H$-function in Figure \ref{figure 1}, we compute the corresponding $HFL^{-}(\L, (s_{1}, s_{2}))$ in each case. Note that we use the convention that the cube grading of $H_{\ast}(A^{-}(s_1, s_2))$ is $0$. For details, see \cite{Liu3}.

\textbf{Case 1:} $i(b)=a, i(c)=b-d, i(d)=a$ and $i(a)=0$, so $HFL^{-}(s_{1}, s_{2})=0$. 

\textbf{Case 2:} $i(b)=a, i(c)=Ub-d, i(d)=Ua$ and $i(a)=0$, so $HFL^{-}(s_{1}, s_{2})=0$. 

\textbf{Case 3:} $i(b)=Ua, i(c)=b-Ud, i(d)=a$ and $i(a)=0$, so $HFL^{-}(s_{1}, s_{2})=0$.

\textbf{Case 4:} $i(b)=a, i(c)=Ub-Ud, i(d)=a$ and $i(a)=0$, so $HFL^{-}(s_{1}, s_{2})=\langle b-d \rangle$. Both $b$ and $d$ have homological grading $-2h$ and cube grading $1$. The Maslov grading of $b-d$ is $-2h+1$. Thus  $HFL^{-}(s_{1}, s_{2})=\F[-2h+1]$. 

\textbf{Case 5:} $i(b)=Ua, i(c)=b-d, i(d)=Ua$ and $i(a)=0$, so $HFL^{-}=\langle a \rangle$ with Maslov grading $-2h$. Thus $HFL^{-}(s_{1}, s_{2})=\F[-2h]$.

\textbf{Case 6:} $i(b)=Ua, i(c)=Ub-Ud, i(d)=Ua$, and $i(a)=0$, so $HFL^{-}(s_{1}, s_{2})=\langle a, b-d \rangle$. Here $a$ has Maslov grading $-2h$ and $b-d$ has Maslov grading $-2(h+1)+1=-2h-1$. Thus $HFL^{-}(s_{1}, s_{2})=\F[-2h]\oplus \F[-2h-1]$.

\begin{figure}[H]
\begin{picture}(100,130)(140,40)

\put(235,60){\framebox(85,40)}
\put(300,68){\makebox(0,0){$h$}}

\put(300,90){\makebox(0,0){$h$}}

\put(255,68){\makebox(0,0){$h+1$}}

\put(255,90){\makebox(0,0){$h$}}

\put(270,50){\makebox(0,0)[I]{Case 4}}

\put(235,120){\framebox(85,40)}
\put(300,128){\makebox(0,0){$h$}}

\put(300,150){\makebox(0,0){$h$}}

\put(255,128){\makebox(0,0){$h$}}

\put(255,150){\makebox(0,0){$h$}}

\put(270, 110){\makebox(0,0)[I]{Case 1}}

\put(335,120){\framebox(85,40)}
\put(405,128){\makebox(0,0){$h+1$}}

\put(405,150){\makebox(0,0){$h$}}

\put(355, 128){\makebox(0,0){$h+1$}}

\put(355, 150){\makebox(0,0){$h$}}

\put(375, 110){\makebox(0,0)[I]{Case 2}}

\put(335,60){\framebox(85,40)}
\put(405,68){\makebox(0,0){$h+1$}}

\put(405,90){\makebox(0,0){$h$}}

\put(355,68){\makebox(0,0){$h+1$}}

\put(355,90){\makebox(0,0){$h+1$}}

\put(375,50){\makebox(0,0)[I]{Case 5}}

\put(445,120){\framebox(85,40)}
\put(510,128){\makebox(0,0){$h$}}

\put(510,150){\makebox(0,0){$h$}}

\put(465,128){\makebox(0,0){$h+1$}}

\put(465,150){\makebox(0,0){$h+1$}}

\put(490, 110){\makebox(0,0)[I]{Case 3}}

\put(445,60){\framebox(85,40)}
\put(515,68){\makebox(0,0){$h+1$}}

\put(510,90){\makebox(0,0){$h$}}

\put(465,68){\makebox(0,0){$h+2$}}

\put(465,90){\makebox(0,0){$h+1$}}

\put(490, 50){\makebox(0,0)[I]{Case 6}}
\end{picture}
\caption{Possible local behaviours of the  $H$-function. \label{figure 1}}

\end{figure}

\section{Proofs of the main results}
\label{sec:proof}
\subsection{Special \texorpdfstring{$L$}{L}-space links are fibered and strongly quasi-positive}
\begin{lemma}
 \label{lemma:hat}
 Suppose that $\L$ is a fibered link and $x$ is a cycle in $\widehat{A}_0(\L,s_{\textup{top}})$ such that $[x]$ is a generator of $\widehat{HFL}_0(\L,s_{\textup{top}})$. Then we have that $\tau(\L)=s_{\textup{top}}$. 
\end{lemma}
\begin{proof}
 Since $\L$ is fibered, from Theorem \ref{thm:fibered} one has $\dim_{\F}\widehat{HFL}(\L,s_{\text{top}})=1$ and the generator is $[x]$. 
 We first want to show that the homology class of $x$ (note that such a class is not $[x]$) is a generator of $\widehat{HF}_0(S^3)\otimes\left(\F_{-1}\oplus\F_0\right)^{\otimes\:n-1}$ in homological grading 0.
 If this is not the case then there exists a $y\in\widehat A_1(\L,s')$ with $s_{\text{top}}<s'$ such that $\widehat\partial y=x$. This necessarily means that we can find an element $z$, and also a $y'\in\widehat A_1(\L,s'-1)$, such that $\widehat\partial z=y+y'$. Otherwise, $\widehat{HFL}(\L, s')$ is non-zero, which is a contradiction. 
 Hence, we write \[0=\widehat\partial(y+y')=x+\widehat\partial y'\:,\] which means $\widehat\partial y'=x$. Since we can iterate this procedure, everytime strictly decreasing the Alexander filtration level, the claim is proved.
 
 In order to prove that $\tau(\L)=s_{\text{top}}$ we need to obstruct the existence of an $x''\in\widehat A_0(\L,s'')$ with $s''<s_{\text{top}}$ and a $y''$ such that $\widehat\partial y''=x+x''$. Suppose that $y''$ lives in an Alexander level strictly bigger than $s_{\text{top}}$; then we can apply the same argument as before and obtain that $y''\in\widehat A_1(\L,s_{\text{top}})$ necessarily. This is impossible because $[x]$ is a generator of $\widehat{HFL}_0(\L,s_{\text{top}})$.
\end{proof}

We can now prove one of the main results in the paper.
\begin{proof}[Proof of Theorem \ref{stypeb}]
 Since $\L$ is a special $L$-space link, there is an $\textbf s=(s_1,...,s_n)\in\H(\L)$ such that $H_{\mathcal L}(\textbf w)=0$ if and only if $\textbf w\succeq\textbf s$, and the $H$-function stabilizes above the hyperplane $\{x_1+...+x_n=s_1+...+s_n-n\}$ which passes through $\bm s-\bm 1$, i.e. 
 for every $\textbf x\in\H(\L)$ and every coordinate $x_i$ such that \[\sum_{j=1}^nx_j\geq\sum_{j=1}^ns_j-n\quad\text{ and }\quad x_i\geq s_i\] one has $H_{\L}(x_1, \cdots, x_{i}, \cdots, x_{n})=H_{\L}(x_{1}, \cdots, \infty, \cdots, x_{n})$.
 
 We claim that $\widehat{HFL}(\L, (s'_{1}, \cdots, s'_{n}))=0$ if $\sum_{j} s'_{j}\geq \sum_{j} s_j$ and $\bm{s}'\neq \bm{s}$. There exists $i$ such that  $s'_{i}>s_{i}$. It suffices to prove that  $HFL^{-}(\L, (s'_{1}, \cdots, s'_{n}))=0$ by Proposition \ref{spectral sequence 2}. The argument is similar to the one in \cite[Theorem 1.3]{Liu}. By Proposition \ref{spectral sequence 1}, there exists a spectral sequence converging to $HFL^{-}(\L, \bm{s}')$ with the $E_{1}$-page $$E_{1}(\bm{s}')=\bigoplus_{B\subset \lbrace 1, \cdots, n \rbrace} H_{\ast}(A^{-}(\L, \bm{s}'-\bm{e}_{B})).$$ 
Let $\mathcal{K}=\lbrace 1, \cdots, n \rbrace \setminus \lbrace i \rbrace$, and 
$$E'(\bm{s}')=\bigoplus_{B\subset \mathcal{K}} H_{\ast}(A^{-}(\L, \bm{s}'-\bm{e}_{B})),\quad  E''(\bm{s}')=\bigoplus_{B\subset \mathcal{K}} H_{\ast}(A^{-}(\L, \bm{s}'-\bm{e}_{B}-\bm{e}_{i})).$$
Then $E_{1}(\bm{s}')=E'(\bm{s}')\oplus E''(\bm{s}')$. Recall that for each $B\subset \lbrace 1, \cdots, n \rbrace$, $H_{\ast}(A^{-}(\L, \bm{s}'-\bm{e}_{B}))\cong \F[U]$. Let  $\partial_{1}, \partial', \partial ''$ denote  the differentials in $E_{1}(\bm{s}'), E'(\bm{s}')$ and $E''(\bm{s}')$, respectively. Let $z$ denote the generator of $H_{\ast}(A^{-}(\L, \bm{s}'-e_{B}-e_{i}))\in E''(\bm{s}')$ with homological grading $-2H(\bm{s}'-e_{B}-e_{i})$. Observe that $H(\bm{s}'-e_{B}-e_{i})=H(\bm{s}'-e_{B})$ since $s'_{i}>s_{i}$ and $\L$ is special. Then $\partial_{1}(z)=\partial''(z)+z'$ where $z'$ is the generator of $H_{\ast}(A^{-}(\L, \bm{s}'-\bm{e}_{B}))$ with homological grading $-2H(\bm{s}'-e_{B})$. Let $\mathcal{D}$ be an acyclic chain complex with two generators $a$ and $b$, and the differential $\partial_{D}(a)=b$. Then the chain complex $(E_{1}(\bm{s}' ), \partial_{1} )$  is isomorphic to $(E'(\bm{s}')\otimes \mathcal{D}, \partial_{1}\otimes \partial_{D})$. Thus $E_{2}=0$, and the spectral sequence collapes at $E_{2}$. Therefore, $HFL^{-}(\L, \bm{s}')=0$ whenever $s_1'+...+s_n'\geq s_1+...+s_n$ and $\textbf s'\neq\textbf s$.

We now prove that $\widehat{HFL}(\L, \bm{s})=\F_{0}$. Since $\L$ is special, $H_{\L}(\bm{s})=0$ and for any non-empty subset $B\subset \{1, \cdots, n\}$, $H_{\L}(\bm{s}-\bm{e}_{B})=1$. Then for the spectral sequence of $HFL^{-}(\L, \bm{s})$, it is not hard to see that there is only one tower with top grading $0$ in the $E_1$-page,  which corresponds to $H_{\ast}(A^{-}(\L, \bm{s}))$ by Proposition \ref{spectral sequence 1}. Let $z_{\bm{s}}$ denote the cycle in $H_{\ast}(A^{-}(\L, \bm{s}))$ with Maslov grading $0$. By Remark \ref{differential}, it is the only cycle sin the $E_2$-page. Hence, the spectral sequence collapses at $E_{2}$-page, and $HFL^{-}(\L, \bm{s})=\F_{0}$ generated by $z_{\bm{s}}$. Recall that $HFL^{-}(\L, \bm{s}')=0$ for all $\bm{s}'=\bm{s}+\bm{e}_{B}$. Hence, 
$$\widehat{HFL}(\L, \bm{s})=HFL^{-}(\L, \bm{s})=\F_{0}$$
by Proposition \ref{spectral sequence 2}. Therefore, for $s=s_{1}+\cdots+ s_{n}$ 
$$\widehat{HFL}(L, s)=\bigoplus_{\{\bm{v}=(v_{1}, \cdots, v_{n})\mid v_{1}+\cdots+v_{n}=s\}}\widehat{HFL}(\L, \bm{v})=\widehat{HFL}(\L, \bm{s})=\F_{0}.$$
By  Theorem \ref{thm:fibered}, the link $\L$ is fibered. By the argument above, we see that $s=s_{\textup{top}}$ and there exists a cycle in  $\widehat{A}_{0}(\L, s_{\textup{top}})$ which is a generator of $\widehat{HFL}_{0}(\L, s_{\textup{top}})$. By Lemma  \ref{lemma:hat},  we conclude that $\tau(\L)=s=g_3(\L)+n-1$. Finally, we can use Theorem \ref{thm:fibered_SQP} and obtain that $\L$ is also strongly quasi-positive.
\end{proof}

\begin{proof}[Proof of Theorem \ref{2com}]
Suppose the $L$-space link $\L$ is fibered and strongly quasi-positive. Then $s_{\text{top}}=\tau$ and $\widehat{HFL}(\L, s_{\text{top}})=\F_{0}$ by Lemma \ref{lemma:hat}. Hence, there exists a lattice point $\bm{s}=(s_{1}, s_{2})$ such that $\widehat{HFL}(\L, \bm{s})=\F_{0}$ and $s_{\text{top}}=s_1+s_{2}$. We claim that $HFL^{-}(\L, \bm{s})=\widehat{HFL}(\L, \bm{s})=\F_{0}$ and $H_{\L}(\bm{s})=0$. In fact, for all $\bm{s}\prec \bm{s}'$, we have $HFL^{-}(\L, \bm{s}')=0$. Otherwise, there exists a lattice point  $\bm{s}\prec \bm{v}$ such that $HFL^{-}(\L, \bm{v})\neq 0$, but $HFL^{-}(\L, \bm{v}')=0$ for all $\bm{v}\prec \bm{v}'$. By Proposition \ref{spectral sequence 2}, $\widehat{HFL}(\L, \bm{v})=HFL^{-}(\L, \bm{v})\neq 0$, contradicting to the assumption of $s_{\textup{top}}$. Hence, we obtain that $HFL^{-}(\L, \bm{s})=\widehat{HFL}(\L, \bm{s})=\F_{0}$ by applying the spectral sequence again. 

If $H_{\L}(\bm{s})\geq 1$, by Lemma \ref{growth control}, $H_{\L}(\bm{s}')\geq 1$ for all $\bm{s}'\prec \bm{s}$. By the computation of $HFL^{-}(\L)$ for 2-component $L$-space links in Section \ref{sec:the H-function}, it is not hard to see that the generator of $HFL^{-}(\L, \bm{s})$ has Maslov grading $\leq -1$, contradicting to the assumption that $\widehat{HFL}(\L, \bm{s})=\F_{0}$. At the end, we prove that the link $\L$ is type (B). Combining the condition $H_{\L}(\bm{s})=0$ and $HFL^{-}(\L, \bm{s})=\F_{0}$, it is not hard to see that the $H$-function corresponding to $HFL^{-}(\L, \bm{s})$ is one of the following two cases:

\begin{figure}[H]
\centering
\includegraphics[width=3.0in]{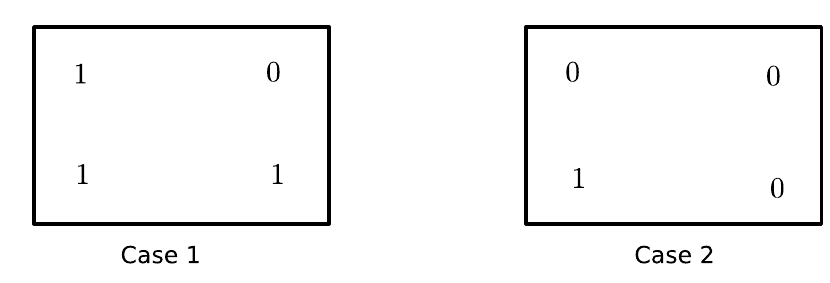}
\caption{In the proof of Theorem \ref{2com} we restricted the local behaviour of the $H$-function to two possible cases.}  
\label{F2}
\end{figure}

However, in Case 2 $HFL^{-}(\L, \bm{s})=\F_{1}$ because of Remark \ref{remark:total}, which contradicts our assumption. Then the $H$-function is as in Case 1. In order to prove the link is type (B), it suffices to prove that $H_{\L}(s_{1}-1, s'_{2})=1$ for all $s'_{2}> s_{2}$ and similarly, $H_{\L}(s'_{1}, s_{2}-1)$ for all $s'_{1}> s_{1}$. If there exists a minimum $s'_{2}>s_{2}$ such that $H_{\L}(s_{1}-1, s'_{2})\neq 1$, by Lemma \ref{growth control}, it must be $0$. By the computation of $HFL^{-}$ in Section \ref{sec:computation},  we  see that $HFL^{-}(\L, (s_{1}, s'_{2}))\neq 0$, contradicting our claim above. Hence, the link is type (B). 

The proof that the link is special requires us to also show that the $H$-function stabilizes, which means it satisfies Equation \eqref{equ3}. If this does not happen, by the similar argument above,   there exists a lattice point $\bm{v}=(v_{1}, v_{2})\in \H(\L)$ such that $HFL^{-}(\L, \bm{v})=\widehat{HFL}(\L, \bm{v})\neq 0$ and $v_{1}+v_{2}\geq s_{1}+s_{2}$. Hence either $\dim\widehat{HFL}(\L,s')>0$ for an $s'>s_{\text{top}}$ or $\dim\widehat{HFL}(\L,s_{\text{top}})>1$; the first claim is impossible, while the second one contradicts the assumption of $\L$ being fibered.
\end{proof}
\begin{proof}[Proof of Corollary \ref{cor:genus}]
 It follows combining the proof of Theorem \ref{stypeb} with the definition of special $L$-space link. In fact, one has that $\textbf s=(s_1,...,s_n)$ is such that \[s_i=g_3(L_i)+\dfrac{\lk(L_i,\L\setminus L_i)}{2}\] and then $s_1+...+s_n=s_{\text{top}}$ as we saw before. Moreover, the fact that $\tau(\L)=\nu^+(\L)=g_3(\L)+n-1$ is a consequence of Theorem \ref{thm:fibered_SQP}.
\end{proof}

\subsection{Algebraic and quasi-positive links}
We recall that algebraic links are the links of planar complex curve singularities in $S^3$. They are coded by their embedded resolution graphs, which are connected and negative definite graphs, see \cite{Neumann1,Neumann2}. In particular, algebraic links are (a subfamily of) iterated positive cables of the unknot and it is known, see \cite[Theorem 2.1.5]{GN2}, that they are $L$-space links.
\begin{proof}[Proof of Proposition \ref{algebralinks}]
 Suppose that $\L=L_1\cup...\cup L_n$ is an algebraic link; 
 let $\bm{s}=(s_{1}, \cdots, s_n)$ such that 
 \[s_{i}=g_{3}(L_{i})+\dfrac{\lk(L_{i}, \L\setminus L_{i})}{2}\:.\]
 We first prove that  
 $$H_{\L}(v_{1}, \cdots, v_{i}, \cdots, , v_{n})=H_{\L}(v_{1}, \cdots, \infty, \cdots, v_{n} ) \quad \textup{ if }v_{i}\geq s_{i}\:. $$
Recall that 
 $$H_{\L}(\infty, \cdots, v_{i}, \cdots, \infty)=0 \quad \textup{ for all } v_{i}\geq s_{i}\:.$$
 Then $H_{\L}(\bm{v})=H_{\L}(\bm{v}+e_{i})=0$ if $v_{i}\geq s_{i}$ and $v_{j}\gg 0$ for all $j\neq i$. By \cite[Lemma 3.6]{GN}, we have
 $$H_{\L}(\bm{v}-e_{K})=H_{\L}(\bm{v}-e_{K}+e_{i})\:$$
 for any subset $K\subset \{1, \cdots, n \} \setminus \{i\}$.
 Iterating this equation, one can prove that 
 $$H_{\L}(v_{1}, \cdots, v_{i}, \cdots, v_{n})=H_{\L}(v_{1}, \cdots, \infty, \cdots, v_{n}) \quad  \textup{ if }  v_{i}\geq s_{i}\:.$$
 By using the similar argument, one can show that the $H$-function satisfies Equation \eqref{equ3}. 
 Now it suffices to prove that $H_{\L}(\bm{s}-\bm{1})=1$. Observe that $H(\bm{s}-e_{i})=1$ for any $i\in \{1, \cdots, n\}$ by \cite[Lemma 3.6]{GN}: then the claim follows from Equation (3.2), Lemma 3.6 and Paragraph 7.1.A in \cite{GN}, where it is shown that if $\bm s\succ\bm v$ and $H_{\L}(\bm v-e_j)>H_{\L}(\bm v)$ for some $j\in\{1,...,n\}$ then $v_j<s_j$.
 
 In fact, take a non-empty $N\subset\{1, \cdots, n \} \setminus \{j\}$. Then, writing $N'=N\cup\{j\}$, we can suppose that $v=\bm s-e_N$ and $v-e_j=\bm s-e_{N'}$. For what we said before we should have $v_j<s_j$, but we see that $v_j=s_j$ and this is a contradiction. It follows that $H_{\L}(\bm s-e_{N'})=H_{\L}(\bm s-e_N)$ and inductively we obtain $H_{\L}(\bm s-\bm 1)=H_{\L}(\bm s-e_{\{1,...,n\}})=H_{\L}(\bm s-e_1)=1$.

\end{proof}

\begin{remark}
Note that the convention for the  $H$-function here  is different from the convention in \cite{GN}. The lattice point $\bm{s}$ in our proof corresponds to the lattice point $(0, \cdots, 0)$ in \cite{GN}, and their $h$-function (this is how it is denoted in \cite{GN}) is obtained from our $H$-function by shifting $\bm{s}$ to the origin and then reflecting at the origin.
\end{remark}

Quasi-positive links are defined as the transverse intersection of the 3-sphere $S^3\subset\mathbb C^2$ with the complex curve $f^{-1}(0)$,  where $f:\mathbb C^2\rightarrow\mathbb C$ is a non-constant polynomial. It follows from their braid equivalent definition (\cite{Cavallo18-2}) that strongly quasi-positive links, and then algebraic links, are quasi-positive, but the latter one is in fact a much larger class.
\begin{proof}[Proof of Proposition \ref{prop:quasi}]
 Let us denote the the 2-bridge link $b(4k^{2}+4k, -2k-1)$ with $L_k$ and its mirror image with $L_k^*$. Since $L_k$ is non-split alternating, we compute easily that \[\tau(L_k)=k\quad\text{ and }\quad\tau(L_k^*)=1-k\:,\] see \cite{Cavallo18}. From \cite{Cavallo18-2} we have that if a 2-component link $L$ is quasi-positive then $2\tau(L)-2=-\chi_4(L)$, where $\chi_4(L)$ is the maximal value of the Euler characteristic of a properly embedded, compact, oriented surface $\Sigma$ in $D^4$ such that $\partial\Sigma=L$. Hence, the link $L_k^*$ cannot be quasi-positive because $\chi_4\leq2$.
 
 In the same way, if $L_k$ is quasi-positive then from \cite[Theorem 1.2]{Cavallo18-2} we have that the maximal self-linking number $\text{SL}(L_k)$ is equal to $2\tau(L_k)-2=2k-2$. Since the link has unknotted components and linking number zero, we can write \[2k-2=\text{SL}(L_k)\leq2\cdot\text{SL}(\bigcirc)=-2\] which implies $k\leq0$ and then we have a contradiction.
\end{proof}


\begin{thebibliography}{99} 



 \bibitem{BBG} M. Boileau, S. Boyer and C. Gordon, \emph{Branched covers of quasipositive links and $L$-spaces}, J. of Topology, \textbf{12} (2019), no. 2, pp. 536--576.
 \bibitem{BG} M. Borodzik and E. Gorsky, \emph{Immersed concordances of links and Heegaard Floer homology},
     Indiana Univ. Math. J., \textbf{67} (2018), no. 3, pp. 1039--1083.
 \bibitem{Cavallo18} A. Cavallo, \emph{The concordance invariant tau in link grid homology}, 
     Algebr. Geom. Topol., \textbf{18} (2018), no. 4, pp. 1917--1951.   
 \bibitem{Cavallo18-2} A. Cavallo, \emph{On Bennequin-type inequalities for links in tight contact 3-manifolds}, J. Knot Theory Ramifications, \textbf{29} (2020), no. 8, 2050055.
 \bibitem{Cavallo19} A. Cavallo, \emph{Locally equivalent Floer complexes and unoriented link cobordisms}, arXiv:1911.03659.
 \bibitem{Cavallo20} A. Cavallo, \emph{Detecting fibered strongly quasi-positive links}, arXiv:2004.02233.
 \bibitem{Neumann1} D. Eisenbud and W. Neumann, \emph{Three-dimensional link theory and invariants of plane curve singularities}, Annals of Mathematics            Studies 110, Princeton University Press, Princeton, NJ, 1985.     
 \bibitem{Ghiggini} P. Ghiggini, \emph{Knot Floer homology detects genus-one fibred knots}, 
     Amer. J. Math., \textbf{130} (2008), no. 5, pp. 1151--1169. 

 \bibitem{Hom} E. Gorsky and J. Hom, \emph{Cable links and $L$-space surgeries}, Quantum Topol., \textbf 8 (2017), no. 4, pp. 629--666. 
 
 \bibitem{GLM} E. Gorsky, B. Liu and A. H. Moore, \emph{Surgery on links of linking number zero and the Heegaard Floer $d$-invariant}, Quantum Topol., \textbf{11} (2020), no. 2, pp. 323--378.
 
  \bibitem{GN} E. Gorsky and A. N\'emethi, \emph{Lattice and Heegaard Floer homologies of algebraic links},
     Int. Math. Res. Not. IMRN, (2015), no. 23, pp. 12737--12780.
     
\bibitem{GN2} E. Gorsky and A. N\'emethi, \emph{On the set of $L$-space surgeries for links}, Adv. Math., \textbf{333} (2018), pp. 386--422. 
 \bibitem{Hedden} M. Hedden, \emph{Notions of positivity and the Ozsv\'ath-Szab\'o concordance invariant}, J. Knot Theory Ramifications, \textbf{19} (2010),     no. 5, pp. 617--629.
 \bibitem{HW} J. Hom and Z. Wu, \emph{Four-ball genus bound and a refinement of the Ozsv\'ath-Szab\'o $\tau$-invariant},
                J. Symplectic Geom., \textbf{14} (2016), no. 1, pp. 305--323. 
 \bibitem{Kawauchi} A. Kawauchi, \emph{A survey of knot theory}, Birkh\"auser Verlag, Basel, 1996, pp. xxii and 420.  
 
 \bibitem{Liu3} B. Liu, \emph{Heegaard Floer homology of $L$-space links with two components}, Pacific J. Math., \textbf{298} (2019), no. 1, pp. 83--112. 
 \bibitem{Liu} B. Liu, \emph{Four-genera of links and Heegaard Floer homology},
     Algebr. Geom. Topol., \textbf{19} (2019), no. 7, pp. 3511--3540.
     \bibitem{Liu2} B. Liu, \emph{$L$-space surgeries on 2-component $L$-space links}, arXiv:1905.04618. 
 \bibitem{LiuY} Y. Liu, \emph{$L$-space surgeries on  links}, Quantum Topol., \textbf{8} (2017), no. 3, pp. 505--570.  
 \bibitem{MO} C. Manolescu and P. Ozsv\'ath, \emph{Heegaard Floer homology and integer surgeries on links},
     arXiv:1011.1317.
 \bibitem{Neumann2} W. Neumann, \emph{A calculus for plumbing applied to the topology of complex surface singularities and degenerating complex curves}, Trans.     Amer. Math. Soc., \textbf{268} (1981), no. 2, pp. 299--344.
 \bibitem{Ni} Y. Ni, \emph{Knot Floer homology detects fibred knots},  Invent. Math., \textbf{170} (2007), no. 3, pp. 577--608.     
 \bibitem{Ni09} Y. Ni, \emph{Link Floer homology detects the Thurston norm}, Geom. Topol., \textbf{13} (2009), no. 5, pp. 2991--3019.     
 \bibitem{OS1} P. Ozsv\'ath and Z. Szab\'o, \emph{Holomorphic disks and topological invariants of closed three-manifolds},
                Ann. of Math. (2), \textbf{159} (2004), no. 3, pp. 1159--1245.     
 \bibitem{OS05} P. Ozsv\'ath and Z. Szab\'o, \emph{On knot Floer homology and lens space surgeries}, 
     Topology,  \textbf{44} (2005), no. 6, pp. 1281--1300.
 \bibitem{OS:linkpoly} P. Ozsv\'ath and Z. Szab\'o, \emph{Link Floer homology and the Thurston norm},
 J. Amer. Math.Soc., \textbf{21} (2008), no. 3, pp. 671--709.     
 \bibitem{SZ3} P. Ozsv\'ath and Z. Szab\'o, \emph{Holomorphic discs, link invariants and the multi-variable Alexander polynomial}, Geom. Topol., {\bf8} (2008), no. 2, pp. 615--692. 
 
 \bibitem{Sar} S. D. Rasmussen, \emph{$L$-space surgeries on satellites by algebraic links}, J. Topol., \textbf{13} (2020), no. 4, pp. 1333--1387. 
  
\end{thebibliography}
\end{document}